\documentclass[11pt]{amsart}
\usepackage{lmodern}
\usepackage{amsmath, amsthm, amssymb, amsfonts,tikz}
\usepackage[normalem]{ulem}
\usepackage{hyperref}

\usepackage{mathrsfs}

\usepackage{verbatim} 
\usepackage{longtable}

\hypersetup{
    colorlinks=true,
    linkcolor=black,
    filecolor=magenta,      
    urlcolor=cyan,
    citecolor=magenta,
    pdftitle={Overleaf Example},
    pdfpagemode=FullScreen,
    }

\usepackage{mathtools}

\usetikzlibrary{decorations.pathmorphing}
\tikzset{snake it/.style={decorate, decoration=snake}}

\usepackage{caption}

\usepackage{tikz-cd}
\usetikzlibrary{arrows}

\theoremstyle{plain}
\newtheorem{thm}{Theorem}[section]
\newtheorem{cor}[thm]{Corollary}
\newtheorem{lem}[thm]{Lemma}
\newtheorem{prop}[thm]{Proposition}
\newtheorem{conj}[thm]{Conjecture}

\theoremstyle{definition}
\newtheorem{defn}[thm]{Definition}
\newtheorem{example}[thm]{Example}

\theoremstyle{remark}
\newtheorem{rmk}[thm]{Remark}

\newcommand{\BC}{{\mathbb{C}}}

\newcommand{\BF}{{\mathbb{F}}}

\newcommand{\BN}{{\mathbb{N}}}

\newcommand{\BP}{{\mathbb{P}}}
\newcommand{\BQ}{{\mathbb{Q}}}

\newcommand{\BZ}{{\mathbb{Z}}}

\newcommand{\CC}{{\mathcal C}}

\newcommand{\CF}{{\mathcal F}}

\newcommand{\CH}{{\mathcal H}}
\newcommand{\CI}{{\mathcal I}}

\newcommand{\CL}{{\mathcal L}}

\newcommand{\CO}{{\mathcal O}}

\newcommand{\codim}{\operatorname{codim}}

\DeclareFontFamily{OT1}{rsfs}{}
\DeclareFontShape{OT1}{rsfs}{n}{it}{<-> rsfs10}{}
\DeclareMathAlphabet{\curly}{OT1}{rsfs}{n}{it}

\newcommand{\wecomment}[1]{{\color{blue}W: #1}}

\usepackage{tikz}
\usepackage{lmodern}
\usetikzlibrary{decorations.pathmorphing}

\begin{document}
\title[Cohomological stabilization, perverse filtrations, and refined BPS invariants]{Cohomological stabilization, perverse filtrations, and refined BPS invariants for del Pezzo surfaces}
\date{\today}

\author[W. Pi]{Weite Pi}
\address{Yale University}
\email{weite.pi@yale.edu}

\author[J. Shen]{Junliang Shen}
\address{Yale University}
\email{junliang.shen@yale.edu}

\author[F. Si]{Fei Si}
\address{BICMR, Peking University}
\email{sifei@bicmr.pku.edu.cn}

\author[F. Zhang]{Feinuo Zhang}
\address{SCMS, Fudan University}
\email{fnzhang21@m.fudan.edu.cn}

\begin{abstract}

We prove an asymptotic product formula for the refined BPS invariants associated with a local del Pezzo surface. Our formula governs the cohomological stabilization of the perverse filtration on the intersection cohomology of the moduli space of 1-dimensional semistable sheaves on a del Pezzo surface. Combined with the theory of Fourier transform of Maulik--Shen--Yin, we show that the perverse filtration matches asymptotically with the Chern filtration defined via tautological classes. In the case of the projective plane, our results resolve conjectures of Kononov--Pi--Shen.



\end{abstract}

\maketitle

\setcounter{tocdepth}{1} 

\tableofcontents
\setcounter{section}{-1}

\section{Introduction}

Throughout, we work over the complex numbers $\mathbb{C}$.

\subsection{Moduli of 1-dimensional sheaves}

Let $(S,L)$ be a polarized del Pezzo surface. Assume that $\beta$ is an ample curve class on $S$ and $\chi\in \mathbb{Z}$. The purpose of this paper is to study cohomological structures of the moduli space $M_{\beta,\chi}$ of semistable 1-dimensional sheaves $\CF$ on $S$ with
\[
[\mathrm{supp}(\CF)] = \beta, \quad \chi(\CF) = \chi.
\]
Here $\mathrm{supp}(-)$ denotes the Fitting support, and the stability condition is with respect to the slope
\[
\mu(\CF) = \frac{\chi(\CF)}{c_1(\CF) \cdot L}.
\]
The geometry of such moduli spaces was studied by Le Potier \cite{LP1} and was connected to enumerative geometry by Katz \cite{Katz}. In recent years, further rich connections have been found between cohomological structures of $M_{\beta,\chi}$ and other directions. For example, the perverse filtration for $M_{\beta,\chi}$ categorifies the refined BPS invariants associated with the local Calabi--Yau threefold 
$\mathrm{Tot}_S(K_S)$ \cite{HST,KL,MT,MS_GT}; the geometry of perverse filtrations for del Pezzo surfaces are mysteriously analogous to the geometry of the (recently resolved) $P=W$ conjecture in non-abelian Hodge theory \cite{dCHM1, MS_PW, HMMS, MSY}; beautiful links have also been found between the Betti numbers of $M_{\beta,\chi}$, relative Gromov--Witten invariants, and scattering diagrams \cite{PB, BFGW, PB2}. In this paper, we provide a complete description of the asymptotic cohomological behavior of $M_{\beta,\chi}$ from both perspectives of numerical invariants and perverse filtrations. Our results prove conjectures in \cite{KPS} concerning the refined BPS invariants and the ``$P=C$'' phenomenon in the case of $\mathbb{P}^2$.

Since $M_{\beta,\chi}$ is singular in general, it is more natural to consider the intersection cohomology $\mathrm{IH}^*(M_{\beta,\chi})$ (with rational coefficients). The Hilbert--Chow morphism
\[
h: M_{\beta,\chi} \to |\beta|, \quad \CF \mapsto \mathrm{supp}(\CF)
\]
further endows the intersection cohomology an increasing filtration --- the perverse filtration
\begin{equation}\label{P0}
P_0\mathrm{IH}^*(M_{\beta,\chi}) \subset P_1 \mathrm{IH}^*(M_{\beta,\chi})\subset  \cdots \subset \mathrm{IH}^*(M_{\beta,\chi}).
\end{equation}

We refer to Section \ref{sec1.1} for generalities on the perverse filtration. The following result of Maulik--Shen \cite{MS_GT} shows that the data (\ref{P0}) does not depend on $\chi \in \BZ$.\footnote{This result was proven in \cite{MS_GT} for toric del Pezzo surfaces, where the toric condition was used to prove a relative dimension bound; this bound was later generalized to all del Pezzo surfaces by the work of Yuan \cite{YY}.} 

\begin{thm}[$\chi$-independence \cite{MS_GT}]\label{thm0.1}
    For any two integers $\chi, \chi' \in \BZ$, we have an isomorphism
    \[
    P_k \mathrm{IH}^m(M_{\beta,\chi}) \simeq P_k \mathrm{IH}^m(M_{\beta,\chi'}).
    \]
\end{thm}

The main character of this paper are the refined BPS invariants
\begin{equation}\label{BPS}
n_\beta^{i,j}: = \dim \mathrm{Gr}^P_i \mathrm{IH}^{i+j}(M_{\beta, \chi}) 
\end{equation}
given by dimensions of the associated graded of the perverse filtration. They do not depend on the choice of $\chi$ by Theorem \ref{thm0.1}. These invariants refine the (intersection) Betti numbers
\begin{equation}\label{refine0}
\dim \mathrm{IH}^{k}(M_{\beta,\chi}) = \sum_{i=0}^k n^{i,k-i}_\beta,
\end{equation}
and conjecturally also refine the Gromov--Witten/Pandharipande--Thomas invariants of the local surface $\mathrm{Tot}_S(K_S)$ in all genus \cite{HST, KL, MT}. The latter is a major open question in the study of Gopakumar--Vafa theory.

\subsection{G\"ottsche's formula and refined BPS invariants} We say that an assignment
\[
\beta \mapsto N(\beta) \in \BZ_{>0}
\]
of a positive integer to an ample curve class is an \emph{asymptotic bound,} if for any given ample curve class $\beta_0$ and $m_0>0$, there exists $d_0>0$ such that $N(d\beta_0) > m_0$ for all $d>d_0$. Roughly, this says that $N(\beta) \to +\infty$ when the ample curve class $\beta$ is positive enough.

Now we describe the asymptotic formula for the invariants (\ref{BPS}). Let $\rho$ be the Picard number of $S$. We recall from the G\"ottsche formula \cite{Go1} that the Betti number $b_k(S^{[n]})$ of the Hilbert scheme of points on $S$ stabilizes to the $q^k$-coefficient of 
\begin{equation}\label{Gott}
T(q):= \prod_{i\geq 0} \frac{1}{(1-q^{2i+2})^{\rho +1}(1-q^{2i+4})}.
\end{equation}
We define an \emph{exotic} refinement of $T(q)$ as follows:
\[
H(q,t):= \frac{1}{(1-qt)^{\rho-1}}\cdot \prod_{i\geq 0} \frac{1}{(1-(qt)^iq^2) (1-(qt)^{i+2})^\rho (1-(qt)^it^2)};
\]
note that this is very different from the natural Hodge polynomial refinement. The formula $H(q,t)$ recovers both the (asymptotic) G\"ottsche formula $T(q) = H(q,q)$ and the product formula of Kononov--Pi--Shen \cite{KPS} when $S = \BP^2$.

\smallskip

Our first main result is the following, which states that $H(q,t)$ completely describes the asymptotic behavior of the refined BPS invariants $n_\beta^{i,j}$.

\begin{thm}[Refined cohomological stabilization]\label{thm0.2}
Let $\beta$ be an ample class on $S$. There exists an asymptotic bound $N(\beta)$ such that 
\begin{equation*}\label{refinedBPS}
n^{i,j}_\beta = [H(q,t)]^{i,j}, \quad {i+j} \leq N(\beta).
\end{equation*}
In particular, we obtain from \eqref{refine0} the cohomological stabilization
\[
\dim \mathrm{IH}^k(M_{\beta,\chi}) = [T(q)]^k, \quad k\leq N(\beta).
\]
Here $[-]^{i,j}$ and $[-]^k$ denote the $q^it^j$-coefficient and the $q^k$-coefficient respectively.
\end{thm}

Cohomological stabilization has been studied by Coskun--Woolf \cite{CW} in positive rank cases. For torsion sheaves, the moduli space admits an extra Hilbert--Chow morphism, refining the (intersection) cohomology via the perverse filtration. Theorem \ref{thm0.2} can be viewed as a torsion sheaf analog of the Coskun--Woolf cohomological stabilization phenomenon.

The asymptotic bound $N(\beta)$ in Theorem \ref{thm0.2} is explicit and relies on the geometry of the linear system $|\beta|$ on the surface $S$; see the formula (\ref{final_bd}) and Remark \ref{final1}. We will discuss this bound further in the case $S=\mathbb{P}^2$ in Section \ref{sec0.4}; in particular, our bound for $\BP^2$ is optimal, proving conjectures of Kononov--Pi--Shen \cite{KPS}. This includes \textit{the $P=C$ conjecture} \cite[Conjecture 0.3]{KPS} which we discuss next.

\subsection{The $P=C$ phenomenon}\label{Sec0.3}
In this section, we assume that the moduli space $M_{\beta,\chi}$ is nonsingular and it admits a universal family. For example, this is the case when 
\[
\mathrm{gcd}(\beta\cdot L, \, \chi) =1.
\]
We may then consider the \emph{ring structure} and \emph{tautological classes} of the rational cohomology $H^*(M_{\beta,\chi})$. Our goal is to determine the perverse filtration for $M_{\beta,\chi}$ using the \textit{Chern filtration} obtained from a universal family. This is analogous to the case of the Hitchin system; see the discussions in Section \ref{sec1.4}.

Following \cite{PS}, we introduce the (normalized) tautological classes
\[
c_k(\gamma) \in H^{*}(M_{\beta,\chi})
\]
given by integrating $\mathrm{ch}_{k + 1}(\mathbb F)$ of a normalized universal family $\mathbb{F}$ over $\gamma \in H^{*}(S)$.
The tautological classes induce an increasing filtration --- the Chern filtration
\[
C_0H^*(M_{\beta,\chi}) \subset C_1H^*(M_{\beta,\chi}) \subset \cdots \subset H^*(M_{\beta,\chi}),
\]
where the $k$-th piece $C_k H^{*}(M_{\beta,\chi})$ is spanned by monomials
\[
\prod_{i=1}^s c_{k_i}(\gamma_i) \in H^{*}(M_{\beta,\chi}), \quad \sum_{i=1}^s k_i \leq k.
\]
We refer to Section \ref{sec1.2} for the precise definition and more details. The following result is a consequence of Theorem~\ref{thm0.2} and the theory of Fourier transform of Maulik--Shen--Yin \cite{MSY}:

\begin{thm}[Asymptotic $P=C$]\label{thm0.3}
Let $N(\beta)$ be the asymptotic bound in Theorem \ref{thm0.2}. Then we have 
\[
P_kH^{\leq N(\beta)}(M_{\beta, \chi}) = C_k H^{\leq N(\beta)}(M_{\beta,\chi}).
\]
\end{thm}


Compared to other del Pezzo surfaces, the case of $\BP^2$ has been more extensively explored in terms of the cohomology of the moduli of 1-dimensional sheaves and the $P=C$ phenomenon. We discuss this special case next.

\subsubsection{Conjectures of Kononov--Pi--Shen for $\BP^2$}\label{sec0.4}

For convenience, we denote by $M_{d,\chi}$ the moduli space associated with the curve class $\beta = dH$ and $\chi\in \mathbb{Z}$, where $H$ is the class of a line. We assume $\mathrm{gcd}(d,\chi)=1$ so that the assumption of Section \ref{Sec0.3} is satisfied. 

Using scattering diagrams, Bousseau \cite{PB} obtained an algorithm to compute the Betti numbers $b_k(M_{d,\chi})$, which links $b_k(M_{d,\chi})$ to Gromov--Witten theory \cite{BFGW}. In \cite{YY, YY2}, Yuan studied the cohomology of $M_{d,\chi}$ using motivic methods; she proved that 
\begin{align*}
    & b_k(M_{d,\chi}) = [T(q)]^k, \quad k\leq 2d-4,\\ 
    & b_{2d-2}(M_{d,\chi}) = [T(q)]^{2d-2} - 3.
\end{align*}
In particular, Yuan's results imply that $2d-4$ is the optimal bound for the cohomological stabilization 
\[
b_k(M_{d,\chi}) = [T(q)]^k
\]
to hold. This stabilization was further refined recently by a conjecture of Kononov--Pi--Shen \cite[Conjecture 0.1]{KPS} concerning the refined BPS invariants $n^{i,j}_d$; it predicts that the refined asymptotic formula also holds with respect to this optimal bound:
\begin{equation}\label{conj_KPS1}
n_d^{i,j} = [H(q,t)]^{i,j}, \quad i+j \leq 2d-4.
\end{equation}

Our bound $N(\beta)$ in the proof of Theorem \ref{thm0.2} indeed achieves the optimal bound $2d-4$; thus we complete the proof of the conjecture (\ref{conj_KPS1}) proposed in \cite{KPS}. Our method also gives a new proof of Yuan's stabilization result \cite[Theorem 1.7]{YY} for the Betti numbers $b_k(M_{d,\chi})$ via the perverse filtration.

\begin{thm}\label{thm0.4}
    The identity \eqref{conj_KPS1} holds.
\end{thm}

Finally, we recall the $P=C$ conjecture for $\BP^2$, which has been formulated on the \textit{total cohomology.} This can be viewed as a del Pezzo analog of the $P=W$ conjecture \cite{dCHM1} in non-abelian Hodge theory, since 
all existing proofs \cite{MS_PW, HMMS, MSY} of the $P=W$ conjecture rely on proving a $P=C$ match on the Hitchin moduli space. We refer to Section \ref{sec1.4} for more discussions on the $P=C$ phenomenon.

\begin{conj}[The $P=C$ conjecture for $\BP^2$ \cite{KPS, KLMP}]\label{conj1} We have
\begin{equation*} \label{P=C} 
 P_k H^{*}(M_{d,\chi}) = C_k H^{*}(M_{d,\chi}).
\end{equation*}
\end{conj}

Conjecture \ref{conj1} was first proposed by Kononov--Pi--Shen \cite{KPS} for $H^{\leq 2d-4}(M_{d,\chi})$, whose main purpose was to {categorify} (\ref{conj_KPS1}). Kononov--Lim--Moreira--Pi \cite{KLMP} then proposed that $P=C$ holds for the \textit{total cohomology} and verified the conjecture for $d\leq 5$ using tautological relations. Yuan \cite{Yuan} proved Conjecture \ref{conj1} for $H^{\leq 4}(M_{d,\chi})$ by intersection theory on Hilbert schemes. As a consequence of Theorem \ref{thm0.4} and Maulik--Shen--Yin \cite{MSY}, we complete the proof of $P=C$ as conjectured in \cite{KPS}.

\begin{thm}\label{mainthm}
Conjecture \ref{conj1} holds for $H^{\leq 2d-4}(M_{d,\chi})$.
\end{thm}

The $P=C$ match for the total cohomology is mysterious and still wide open. So far, the best evidence seems to be the proof of this conjecture for $d\leq 5$ in \cite{KLMP}. Nevertheless, if Conjecture \ref{conj1} were to hold, it is also natural to expect that $P=C$ holds for the total cohomology $H^*(M_{\beta,\chi})$ associated with any del Pezzo surface $S$.

\subsection{Acknowledgements}
We would like to thank Ben Davison, Jun Li, Zhiyuan Li, Woonam Lim, Davesh Maulik, Anton Mellit, Miguel Moreira, Georg Oberdieck, Christian Schnell, Claire Voisin, and Longting Wu for interesting and helpful discussions on relevant topics.  J.S.~gratefully acknowledges the hospitality of the Isaac Newton Institute at Cambridge during his stay in the spring of 2024 where part of this paper was written. J.S.~was supported by the NSF grant DMS-2301474, a Sloan Research Fellowship, and a Simons Fellowship during his visit at the Isaac Newton Institute. F.S.~was supported by the NSFC grant 12201011.

\section{Perverse filtrations and Chern filtrations}

In view of Theorem \ref{thm0.1}, from now on we may work only with the moduli space $M_{\beta,\chi}$ that satisfies the assumption of Section \ref{Sec0.3}: it is nonsingular and admits a universal family $\BF$.

\subsection{Perverse filtrations}\label{sec1.1} Let $f: X \to Y$ be a proper morphism between irreducible nonsingular quasiprojective varieties. We assume that $\dim X = a$, $\dim Y = b$, and the morphism $f$ has equal-dimensional fibers of dimension $r=a-b$. The perverse filtration 
\[
P_0H^m(X) \subset P_1H^m(X) \subset \dots \subset P_{2r}H^m(X) =  H^m(X)
\]
is an increasing filtration on the cohomology of $X$ governed by the topology of the morphism $f$. It is defined to be
\[
P_kH^m(X) := \mathrm{Im}\left\{ H^{m-b}(Y, {^\mathbf{p}\tau_{\leq k}} (Rf_* \BQ_X[b])) \to H^{m-b}(Y, Rf_* \BQ_X[b])\right\} \subset H^m(X)
\]
where $^\mathbf{p}\tau_{\leq \bullet }$ is the perverse truncation functor \cite{BBD}. If we apply the decomposition theorem to $f: X\to Y$, we obtain that
\begin{equation}\label{DT0}
R\pi_*\BQ_X[b]\simeq \bigoplus_{i=0}^{2r}\mathcal{P}_i[-i] \in D^b_c(Y)
\end{equation}
with $\mathcal{P}_i$ a semisimple perverse sheaf on $Y$. The perverse filtration can be identified as
\[
P_kH^m(X)=\mathrm{Im}\Big\{H^{m-b}(Y, \bigoplus_{i=0}^k\mathcal{P}_i[-i])\to H^m(X)\Big\}.
\]

 We recall the following standard lemma, which follows for example from the description \cite{dCM} of the perverse filtration.

\begin{lem}\label{lem1.1}
If $\alpha \in H^*(X)$ lies in the kernel of the restriction map $H^m(X) \to H^m(X_y)$ with $X_y:= f^{-1}(y)$ a closed fiber, then we have $\alpha \in P_{m-1}H^m(X)$.
\end{lem}

In this paper, we mainly consider the Hilbert--Chow morphism $h: M_{\beta,\chi} \to |\beta|$; the refined BPS invariants are the dimensions of the cohomology of the perverse sheaves that appeared in the decomposition theorem (\ref{DT0}) associated with $h$; see also Section \ref{3.4.1}.

\subsection{Chern filtrations}\label{sec1.2}

The choice of a universal family over $S\times M_{\beta,\chi}$ is not unique. To introduce well-defined tautological classes $c^D_k(\gamma)$ with $\gamma \in H^*(S)$, we shall normalize our universal family suitably. Unlike the case of $\BP^2$ studied in \cite{PS}, for general del Pezzo surfaces the definition of the normalized tautological classes involves a choice of a divisor class $D$. After $D$ is chosen, the classes $c^D_k(\gamma)$ do not depend on the choice of $\BF$. A correct normalization is crucial for the match between the Chern filtration and the perverse filtration to hold.

Before discussing the case of general del Pezzo surfaces, we recall the case of $\BP^2$ first \cite{PS}. By \cite{Beau}, there is no odd cohomology for $M_{d,\chi}$. The K\"unneth decomposition yields
\[
H^2(\BP^2 \times M_{d,\chi}) = H^2(\BP^2) \oplus H^2(M_{d,\chi}).
\]
Hence any class $\alpha \in H^2(\BP^2\times M_{d,\chi})$ can be written as
\[
\alpha = \alpha_P + \alpha_M,
\]
where $\alpha_\bullet$ denotes the component that is pulled back from the corresponding factor. We define the twisted Chern character associated with a universal family $\BF$ and a class $\alpha \in H^2(\BP^2\times M_{d,\chi})$ as follows:
\[
\mathrm{ch}^\alpha(\BF):= \mathrm{ch}(\BF) \cdot \mathrm{exp}(\alpha) \in H^*(\BP^2 \times M_{d,\chi}).
\]
We denote by $\mathrm{ch}_k^\alpha(\BF) \in H^{2k}(\BP^2\times M_{\beta,\chi})$ its degree $k$ part. The normalized universal class is defined to be the twisted Chern character
\[
\mathrm{ch}^\alpha(\BF) \in H^*(\BP^2 \times M_{d,\chi})
\]
satisfying the normalization conditions:
\begin{equation}\label{normalization}
\int_H \mathrm{ch}_2^\alpha(\BF) =0, \quad \int_{\mathbf{1}_{\BP^2}} \mathrm{ch}_2^\alpha(\BF) = 0.
\end{equation}
Here $\int_{\alpha}(-)$ stands for $\pi_{M*}(\pi_P^*\alpha \cdot (-)) \in H^*(M_{d,\chi})$ with $\pi_\bullet$ the projections. These conditions determine $\alpha$ uniquely, and the resulting tautological classes
\[
c_k(\gamma): = \int_{\gamma} \mathrm{ch}_{k+1}^\alpha(\BF) \in H^*(M_{d,\chi}), \quad \gamma \in H^*(\BP^2)
\]
do not depend on the choice of the pair $(\BF, \alpha)$; see \cite[Section 2]{PS}. The Chern filtration defined via $c_k(\gamma)$ is expected to match perfectly with the perverse filtration \cite{KPS, KLMP} as stated in Conjecture \ref{conj1}.

Now we consider a general del Pezzo surface $S$, and a twisted Chern character as above
\[
\mathrm{ch}^\alpha(\BF) \in H^*(S\times M_{\beta,\chi}), \quad \alpha = \alpha_S + \alpha_M.
\]
By dimension reasons, the conditions (\ref{normalization}) may not be sufficient to determine $\alpha$ from a universal family $\BF$. To fix this, we impose an additional condition on $\alpha_S$. Consider a divisor class
\begin{equation}\label{D}
D \in H^2(S), \quad D\cdot \beta \neq 0.
\end{equation}

\begin{defn}\label{def}
    Fix $D$ as in (\ref{D}). We say that $\mathrm{ch}^\alpha(\BF)$ associated with $(\BF, \alpha)$ is a $D$-normalized universal class, if the class $\alpha_S\in H^2(S)$ is proportional to $D$, and
    \begin{equation}\label{nnormalization}
    \int_{\mathbf{1}_S} \mathrm{ch}_2^\alpha(\BF) = 0, \quad  \int_{D} \mathrm{ch}_2^\alpha(\BF) =0.
    \end{equation}
\end{defn}

When $S=\BP^2$, the class $D$ has to be proportional to the hyperplane class $H \in H^2(\BP^2)$. Therefore the normalized universal class is unique. For other del Pezzo surfaces, however, the normalized universal class depends on the choice of $D$.\footnote{Nevertheless, we will show in Proposition \ref{prop1.4:chernindep} that this choice does not affect the eventual Chern filtration.}

\begin{prop}\label{prop1.3}
    Assume that $\mathrm{ch}^\alpha(\BF)$ is $D$-normalized with respect to a fixed $D$ as in (\ref{D}). Then the following properties hold.
    \begin{enumerate}
        \item[(i)] The class $\alpha = \alpha_S + \alpha_M$ is uniquely determined by $\BF$.
        \item[(ii)] The class $\mathrm{ch}^\alpha(\BF)$ does not depend on the choice of a universal family $\BF$.
        \item[(iii)] We have
        \[
        [\mathrm{ch}_2^\alpha(\BF)]_{(2,2)} \in P_1H^4(S\times M_{\beta,\chi}).
        \]
        Here the perverse filtration is induced by $h_S: S\times M_{\beta,\chi} \to S\times |\beta|$, and $[-]_{(i,j)}$ stands for the K\"unneth factor of a class in $H^i(S) \otimes H^j(M_{\beta,\chi}) \subset H^{i+j}(S\times M_{\beta,\chi})$.
\end{enumerate}
\end{prop}

\begin{proof}
We first note the following facts concerning a universal family $\BF$ over $S\times M_{\beta,\chi}$:
\begin{enumerate}
    \item[(a)] We have \[
[\mathrm{ch}_1(\BF)]_{(2,0)}=\beta, \quad [\mathrm{ch}_1(\BF)]_{(0,2)} = h^*H_{|\beta|},
\]
where $H_{|\beta|}$ is the hyperplane class of the projective space $|\beta|$.
\item[(b)] Let $C\subset S$ represent a nonsingular curve lying in $|\beta|$. Then we have
\[
[\mathrm{ch}_2(\BF)]_{(2,2)}\Big{|}_{S\times h^{-1}([C])} \in \BQ[\beta] \otimes H^2(h^{-1}([C])).
\]
\end{enumerate}
(a) follows from the fact that $\mathrm{ch}_1(\BF)$ calculates the class of the support of the torsion sheaf $\BF$, which is expressed as
\[
[\mathrm{supp}(\BF)] = \pi_S^*\beta + \pi_M^* h^* H_{|\beta|}.
\]
For (b), we notice that the fiber $h^{-1}([C])$ is isomorphic to the Jacobian variety $J_C$; the sheaf $\BF$ is given by the pushforward of a line bundle $\CL$ on $C\times J_C$ along the natural inclusion
\[
\iota: C\times J_C \hookrightarrow S\times J_C.
\]
Therefore the $(2,2)$-component of $\mathrm{ch}_2(\BF)$ is given by the pushforward of a class on $C\times J_C$ of the K\"unneth type $(0,2)$. Its pushforward has to be of the form $\BQ[\beta] \otimes H^2(J_C)$ since $[C] = \beta$.

Now we prove (i). By (a), we have
\[
[\mathrm{ch}^\alpha_2(\BF)]_{(4,0)} = [\mathrm{ch}_2(\BF)]_{(4,0)}+ (\alpha_S\cdot \beta)\otimes \mathbf{1}_{M}.
\]
The second equation of (\ref{nnormalization}) determines the intersection number $\alpha_S \cdot \beta$. Since $\alpha_S$ is required to be proportional to $D$ and $D\cdot \beta\neq 0$, this determines $\alpha_S$ completely. We further consider
\begin{equation}\label{(2,2)}
[\mathrm{ch}^\alpha_2(\BF)]_{(2,2)} = [\mathrm{ch}_2(\BF)]_{(2,2)} + \beta \otimes \alpha_M + \alpha_S \otimes h^*H_{|\beta|}.
\end{equation}
It is clear that the first equation of (\ref{nnormalization}) determines $\alpha_M$ by applying $\int_D(-)$ to (\ref{(2,2)}). This completes the proof of (i).

(ii) is straightforward, and we omit the proof.

To see (iii), we note that 
\[
P_1H^*(S \times M_{\beta,\chi}) = H^*(S) \otimes P_1H^*(M_{\beta,\chi}).
\]
Hence by Lemma \ref{lem1.1} the vanishing over a nonsingular fiber
\[
[\psi]_{(2,2)}\Big{|}_{S\times h^{-1}([C])} =0
\]
for a class $\psi \in H^4(S\times M_{\beta,\chi})$ implies that 
\[
[\psi]_{(2,2)} \in P_1H^4(S\times M_{\beta,\chi}).
\]
We consider the expression
\[
\mathrm{ch}_2^\alpha(\BF) = \mathrm{ch}_2(\BF) + \beta \otimes \alpha_M + (\alpha_S\cdot \beta) \otimes \mathbf{1}_{M} + \Psi
\]
obtained from (a) above, where $\Psi$ is a class satisfying  
\[
\Psi\Big{|}_{S\times h^{-1}([C])} = 0.
\]
By the discussion above, we know that
\[
[\Psi]_{(2,2)} \in P_1H^4(S\times M_{\beta,\chi}).
\]
Therefore, for the same reason, it suffices to show that 
\begin{equation}\label{nor2}
[\mathrm{ch}_2(\BF)]_{(2,2)}\Big{|}_{S\times h^{-1}([C])}  + \beta \otimes \alpha_M\Big{|}_{S\times h^{-1}([C])}  = 0
\end{equation}
By (b), we may assume that 
\[
[\mathrm{ch}_2(\BF)]_{(2,2)}\Big{|}_{S\times h^{-1}([C])} = \beta \otimes \epsilon;
\]
so the left-hand side of (\ref{nor2}) can be written as
\[
\mathrm{LHS}~~\mathrm{of}~~(\ref{nor2}) = \beta \otimes \left(\epsilon + \alpha_M\Big{|}_{S\times h^{-1}([C])}\right).
\]
The condition $\int_D \mathrm{ch}^\alpha_2(\BF) = 0$ from (\ref{nnormalization}) then forces that 
\[
\int_D \beta \otimes \left(\epsilon + \alpha_M\Big{|}_{S\times h^{-1}([C])}\right) = 0,
\]
which completes the proof of (\ref{nor2}).
\end{proof}

For a fixed $D$ satisfying \eqref{D}, we define the tautological classes
\[
c^D_k(\gamma): = \int_\gamma \mathrm{ch}_{k+1}^\alpha(\BF) \in H^*(M_{\beta,\chi})
\]
from the $D$-normalized universal class $\mathrm{ch}^\alpha(\BF)$. This allows us to define the Chern filtration as in Section \ref{Sec0.3}:
\[
C^D_0H^*(M_{\beta,\chi}) \subset C^D_1H^*(M_{\beta,\chi}) \subset \cdots \subset H^*(M_{\beta,\chi}).
\]
Due to the tautological generation result \cite{Beau}, this filtration exhausts the total cohomology $H^*(M_{\beta,\chi})$. The construction of the Chern filtration \emph{a priori} depends on the choice of the divisor $D$; the next proposition states that the Chern filtration in fact does not depend on this choice.

\begin{prop}
\label{prop1.4:chernindep}
For two divisor classes $D_1, D_2$ satisfying \eqref{D}, we have
\[
C^{D_1}_\bullet H^*(M_{\beta,\chi}) = C^{D_2}_\bullet H^*(M_{\beta,\chi}).
\]
\end{prop}

\begin{proof}
    This is a consequence of the fact that the ($D$-normalized) Chern filtration coincides with the Chern filtration defined via a certain weight zero descendent algebra $\mathbb{D}_{(\beta,\chi),\mathrm{wt_0}}$. We refer to \cite[Section 1.1]{KLMP} for the precise definition of $\mathbb{D}_{(\beta,\chi),\mathrm{wt_0}}$\footnote{Although \cite{KLMP} treats only $S=\mathbb{P}^2$, the definition and properties apply identically to  other del Pezzo surfaces.}; roughly speaking, it is a formal algebra of descendents whose cohomological realization does \textit{not} depend on the choice of a universal family. Indeed, for a given divisor class $D_1$ satisfying \eqref{D}, there exists a canonical isomorphism
    \[
    \mathbb{D}_{(\beta,\chi),\mathrm{wt_0}} \simeq \mathbb{Q}[c^{D_1}_k(\gamma)]_{k\geq 0,\, \gamma \in H^*(S)}
    \]
    constructed using the normalization conditions \eqref{nnormalization}; see \cite[Section 1.2.3]{KLMP}. We can also endow the descendent algebra with a Chern filtration $\widetilde{C}_\bullet \mathbb{D}_{(\beta,\chi),\mathrm{wt_0}}$, which further induces a commutative diagram \cite[Proposition 6.5]{KLMP}
    \[\begin{tikzcd}
	{\widetilde{C}_\bullet \mathbb{D}_{(\beta,\chi),\mathrm{wt_0}} } & {C^{D_1}_\bullet \mathbb{Q}[c^{D_1}_k(\gamma)]} \\
	{\widetilde{C}_\bullet H^*(M_{\beta,\chi})} & {C^{D_1}_\bullet H^*(M_{\beta,\chi}).}
	\arrow["\sim", from=1-1, to=1-2]
	\arrow[two heads, from=1-1, to=2-1]
	\arrow[two heads, from=1-2, to=2-2]
	\arrow["\sim", from=2-1, to=2-2]
\end{tikzcd}\]
that identifies the two Chern filtrations in the bottom row. This applies to $D_2$ as well; thus the claim follows by composing the two identifications.
\end{proof}

From now on, we fix a divisor class $D\in H^2(S)$ satisfying (\ref{D}) and use $c_k(\gamma)$, $C_\bullet H^*(M_{\beta,\chi})$ to denote $c^D_k(\gamma)$, $C^D_\bullet H^*(M_{\beta,\chi})$ respectively for notational convenience.

\subsection{The $P\supset C$ inclusion}

We denote by $|\beta|^\circ \subset |\beta|$ the open subset of \emph{integral} curves in the linear system. We define 
\begin{equation*}
N_1(\beta):= 2\cdot \mathrm{codim}_{|\beta|}(|\beta| \setminus |\beta|^\circ)-2.
\end{equation*}

\begin{prop}\label{prop1.4}
    The assignment $\beta \mapsto N_1(\beta)$ is an asymptotic bound.
\end{prop}

\begin{proof}
For a fixed ample curve class $\beta$, we need to show that 
\[
\codim_{|d\beta|}(|d\beta|\setminus|d\beta|^{\circ}) \to +\infty
\]
as $d\to+\infty$. Note that there are two types of curves in $|d\beta|\setminus|d\beta|^{\circ}$:
\begin{enumerate}
\item[(i)] curves that can be written as $C_1+C_2$, where $C_j\in|\beta_j|$ ($j=1,2$) for nef classes $\beta_j$ with $d\beta=\beta_1+\beta_2$;
\item[(ii)] curves that can be written as $D_1+D_2$, where $D_1\in|\nu|$ for a nef class $\nu$ and $D_2\in|d\beta-\nu|$ is a curve whose (reduced) irreducible components are $(-1)$-curves.
\end{enumerate}
The locus of curves in $|d\beta|$ of the above two types are denoted respectively by $Z_1$ and $Z_2$. It suffices to show 
\[
\codim_{|d\beta|}(Z_j) \to +\infty
\]
as $d\to +\infty$.

For $\beta_j$ as in (i), we have
\begin{equation}
\label{eq:beta12}
    d^2\beta^2=\beta_1^2+\beta_2^2+2\beta_1\cdot\beta_2.
\end{equation}
If $\beta_1^2=0$ or $\beta_2^2=0$, without loss of generality we may assume $\beta_1^2=0$. Then
\begin{equation*}
\beta_1\cdot\beta_2=\beta_1(d\beta-\beta_1)=d\beta\cdot\beta_1\geq d.        
\end{equation*}
If $\beta_1^2>0$ and $\beta_2^2>0$, then by the Hodge index theorem and (\ref{eq:beta12}),  $$\beta_1\cdot\beta_2\geq\sqrt{\beta_1^2\beta_2^2}\geq\sqrt{d^2\beta^2-2\beta_1\cdot\beta_2-1},$$
which implies that
\begin{equation*}
    \beta_1\cdot\beta_2\geq d\sqrt{\beta^2}-1.
\end{equation*}
This bound for $\beta_1\cdot \beta_2$, together with the Riemann--Roch formula, yields
\begin{equation*}
\begin{aligned}
    \dim|d\beta|&=\frac{(\beta_1+\beta_2)(\beta_1+\beta_2-K_S)}{2}=\dim|\beta_1|+\dim|\beta_2|+\beta_1\cdot\beta_2\\
    &\geq \dim|\beta_1|+\dim|\beta_2|+\min\{d,d\sqrt{\beta^2}-1\}.
\end{aligned}
\end{equation*}
Since there are finitely many choices for $(\beta_1,\beta_2)$ and any curve in $Z_1$ is in the image of $|\beta_1|\times|\beta_2|$, it follows that
\begin{equation}
    \label{ineq:Z1}
    \codim_{|d\beta|}(Z_1)\geq\min\{d,d\sqrt{\beta^2}-1\}.
\end{equation}

For $\nu$ as in (ii), by the Riemann--Roch formula
\begin{equation*}
    \begin{aligned}
        \dim|d\beta|-\dim|\nu|&=\frac{(d\beta-\nu)d\beta+(d\beta-\nu)\nu-K_S\cdot(d\beta-\nu)}{2}\\
        &\geq\frac{(d\beta-\nu)(d\beta-K_S)}{2}\geq\frac{d+1}{2}.
    \end{aligned}
\end{equation*}
Since there are finitely many choices for such nef class $\nu$ and $(-1)$-curves, we deduce that
\begin{equation}
    \label{ineq:Z2}
    \codim_{|d\beta|}(Z_2)\geq\frac{d+1}{2}.
\end{equation}
The result follows from (\ref{ineq:Z1}) and (\ref{ineq:Z2}).
\end{proof}

\begin{example}\label{ex}
The following examples are obtained from direct calculations, \emph{cf.} \cite[Lemma 5.2]{MSY}. 
\begin{enumerate}
    \item[(i)] When $S=\BP^2$, we have 
    \[
    N_1(dH) = 2d-4, \quad d>0.
    \]
    \item[(ii)] When $S = \BP^1\times \BP^1$, we have
    \[
    N_1(a H_1 + b H_2) = 2\cdot \mathrm{min}\{a,b\} -2, \quad a,\,b >0.
    \]
    Here $H_1$ and $H_2$ denote $\mathbb{P}^1\times \mathsf{pt}$ and $\mathsf{pt} \times \mathbb{P}^1$, respectively.

    \item[(iii)] When $S = \mathrm{Bl}_\mathsf{pt}(\mathbb{P}^2)$, we have
    \[
    N_1(aH - bE) = 2\cdot \min\{a-b,\, b+1\}-2, \quad a>b>0.
    \]
    Here $H$ is the  pullback of a line on $\BP^2$ and  $E$ is the exceptional divisor.
\end{enumerate}   
\end{example}

The next result is a consequence of the theory of Fourier transform established in \cite{MSY}.

\begin{thm}[\emph{cf.} {\cite[Theorem 0.6]{MSY}}]\label{thm1.6}
  We have
    \[
    P_kH^{\leq N_1(\beta)}(M_{\beta,\chi}) \supset C_k H^{\leq N_1(\beta)}(M_{\beta,\chi}).
    \]
\end{thm}

\begin{proof}
    This was proven in \cite[Theorem 0.6]{MSY} in the case $S= \BP^2$. The general case follows from a parallel argument. We sketch the main steps for the reader's convenience.
    
    Let $\CC_{|\beta|^\circ} \to |\beta|^\circ$ be the universal family of integral curves. The restriction of $h: M_{\beta,\chi} \to |\beta|$ to the open subset $|\beta|^\circ$ is given by a compactified Jacobian fibration
    \[
    h^\circ: \overline{J}^e_{\CC_{|\beta|^\circ}} \to |\beta|^\circ, 
    \]
where the degree $e$ is determined by $\chi$ via a Riemann--Roch calculation. The morphism $h^\circ$ endows the cohomology of $\overline{J}^e_{\CC_{|\beta|^\circ}}$ with a perverse filtration. By the argument of \cite[Corollary 5.3]{MSY}, the restriction morphism 
\[
\mathrm{res}: H^{\leq N_1(\beta)}(M_{\beta,\chi}) \to H^{\leq N_1(\beta)}(\overline{J}^e_{\CC_{|\beta|^\circ}})
\]
is an isomorphism preserving the perverse filtrations. 

It was proven in \cite{MSY} (via a study of the convolution product associated with the Fourier transform) that the perverse filtration $P_\bullet H^*(\overline{J}^e_{\CC_{|\beta|^\circ}})$ is multiplicative with respect to the cup product. Therefore we only need to treat a single tautological class $c^D_k(\gamma)$. It suffices to prove 
\begin{equation}\label{ch_k}
\mathrm{ch}^\alpha_{k+1}(\BF) \in P_kH^{2k+2}\left(S\times \overline{J}^e_{\CC_{|\beta|^\circ}}\right)
\end{equation}
for the $D$-normalized universal class. Here the perverse filtration is defined by the morphism $h^\circ_S: S\times \overline{J}^e_{\CC_{|\beta|^\circ}} \to S\times |\beta|^\circ$.

The left-hand side of (\ref{ch_k}) can be calculated in terms of the Fourier transform; see \cite[Proposition 5.1]{MSY}. In order to deduce 
the desired perversity bound for $\mathrm{ch}^\alpha_{k+1}(\BF)$, we only need
\[
[\mathrm{ch}^\alpha_2(\BF)]_{(2,2)} \in P_1H^4\left(S\times \overline{J}^e_{\CC_{|\beta|^\circ}}\right)
\]
as noted in the proof of \cite[Theorem 0.6]{MSY}. This is given by Proposition \ref{prop1.3} (iii).
\end{proof}

\subsection{Remarks on $P=C$}\label{sec1.4}

Before we start to prove our main theorems, we conclude this section with some remarks concerning ``$P=C$" phenomena that appeared in several geometries and strategies to approach them.

The term ``$P=C$" stands for a phenomenon that two structures of very different flavors match on the cohomology of certain moduli of stable sheaves --- the perverse filtration associated with a natural proper map carried by the moduli space, and the Chern filtration obtained from a universal family and tautological classes. The perverse filtration is abstract but the Chern filtration is explicit.

To the best of our knowledge, the $P=C$ phenomenon was first found for the Hitchin system \cite{dCHM1}, and served as a key step in the proofs of the $P=W$ conjecture in non-abelian Hodge theory \cite{MS_PW, HMMS, MSY}. More precisely, the $P=W$ conjecture was reduced to $P=C$ for the Hitchin moduli space $M_{\mathrm{Dol}}$:
\begin{equation}\label{P=W}
P_k H^*(M_{\mathrm{Dol}}, \BQ) = C_k H^*(M_{\mathrm{Dol}}, \BQ).
\end{equation}
There are two approaches to prove (\ref{P=W}). The approach of \cite{MS_PW, MSY} is to prove first the weaker statement $P \supset C$. By the results of Shende \cite{Shende} and Mellit \cite{Mellit} on the character variety, we know that the right-hand side of (\ref{P=W}), which is identified with the weight filtration $W_{2k}H^*(M_B, \BQ)$ of the character variety, satisfies the Lefschetz symmetry. Therefore the weaker $P \supset C$ is sufficient to deduce (\ref{P=W}). The approach of \cite{HMMS} does not rely on Mellit's Lefschetz symmetry \cite{Mellit}; instead, the authors constructed a \emph{splitting} of the perverse filtration using the cohomological Hall algebra, and placed the tautological classes in the desired pieces of this decomposition. In other words, the $P=W$ or the $P=C$ match for the Hitchin system is stronger than a match of two filtrations --- it is indeed a match of two multiplicative decompositions.

Similar $P=C$ phenomena have been discovered for Lagrangian fibrations. Let $S$ be a K3 or an abelian surface, the moduli of 1-dimensional stable sheaves on $S$ admits a Hilbert--Chow morphism, which is Lagrangian and is now referred to as the Beauville--Mukai system. Using Markman's monodromy symmetry \cite{Markman2, Markman3}, it was proven in \cite[Theorem 2.1]{dCMS} that $P=C$ holds. This result was further applied in \cite{dCMS} to prove the $P=W$ conjecture for genus 2 curves via a degeneration argument. Note that in the compact Lagrangian setting, $P=C$ is again a match of two decompositions as in the Hitchin system case.

Compared to the cases above, the $P=C$ match for a del Pezzo surface seems to be quite different. The perverse filtration associated with the Hilbert--Chow morphism $h$ does not admit a multiplicative splitting.\footnote{In the case $S=\BP^2$, this can be seen easily from the obstruction on the existence of a multiplicative splitting of the perverse filtration used in \cite{BMSY}, or from the presentations of the cohomology rings in low degrees \cite{KLMP}.} Therefore, it seems hard to upgrade Conjecture \ref{conj1} to a match of two decompositions. On the other hand, it is in general challenging to prove a match of two \emph{honest} filtrations --- the lack of natural splittings of the filtrations provides difficulty in proving the equality of the two filtrations \emph{directly}. For example, on the perverse filtration side, for a class $\gamma \in P_kH^*(M_{\beta,\chi})$, it is in general hard to prove a statement like $\gamma \not\in P_{k-1}H^*(M_{\beta,\chi})$; on the Chern filtration side, the precise location of a class is significantly influenced by the relations between the tautological classes.  

In this paper, our approach to $P=C$ for del Pezzo surfaces is \emph{numerical}.\footnote{However, it seems hard to generalize the numerical argument of this paper to prove the full $P=C$ conjecture (Conjecture \ref{conj1}); new ideas may be needed.} More precisely, the following numerical formula (\ref{assumption_m}) governs  simultaneously tautological relations and the $P=C$ conjecture for $H^{\leq N_1(\beta)}(M_{\beta,\chi})$.  


\begin{cor}\label{cor1.7}
For a positive integer $m$ satisfying $2m\leq N_1(\beta)$, assuming
\begin{equation}\label{assumption_m}
n^{i,j}_\beta = [H(q,t)]^{i,j}, \quad i+j \leq 2m,
\end{equation}
then the following statements hold for any divisor $D$ satisfying \eqref{D}.
\begin{enumerate}
\item[(i)] For any classes
\[
\mathbf{1}_S \in H^0(S),~~ D, \gamma_1, \cdots, \gamma_{\rho-1} \in H^2(S),~~ \mathsf{pt} \in H^4(S)
\]
that form a basis of $H^*(S)$, the ($D$-normalized) tautological classes
\begin{align*}
   & c_0(\mathsf{pt}),~~ c_1(\gamma_i), ~~ c_2(\mathbf{1}_S) \in H^2(M_{\beta,\chi}),\\
  & c_{k-2}(\mathsf{pt}),~~ c_{k-1}(D), ~~c_{k-1}(\gamma_i),~~ c_{k}(\mathbf{1}_S) \in H^{2k-2}(M_{\beta,\chi}),\\
  &  k\in \{3,4,\dots, m+1\},\quad i \in \{1,2,\dots,\rho-1\}
\end{align*}
have no relations in $H^{\leq 2m}(M_{\beta,\chi})$. 
\item[(ii)] We have
    \[
P_kH^{\leq 2m}(M_{\beta,\chi}) = C_kH^{\leq 2m}(M_{\beta,\chi}).
\]
\end{enumerate}
\end{cor}

\begin{proof}
The proof of (i) resembles that of \cite[Theorem 1.2 (b)]{PS}. By \cite{Beau}, the tautological classes 
\[
c_k(\mathbf{1}_S), ~~ c_k(D), ~~c_k(\gamma_i),~~c_k(\mathsf{pt}) \in H^*(M_{\beta,\chi}), \quad  k\in \BN, \;\; i\in \{1,2,\dots, \rho-1\}
\]
generate the cohomology $H^*(M_{\beta,\chi})$ as a $\BQ$-algebra. The $D$-normalization condition of Definition~\ref{def} forces 
\[
c_1(\mathbf{1}_S) =0,\quad c_1(D) =0.
\]
The remaining non-trivial tautological classes are listed in (i) above. Therefore by counting monomials, we obtain
\begin{equation}\label{assumption_b}
b_i(M_{\beta,\chi}) \leq [T(q)]^i, \quad i\leq 2m.
\end{equation}
The equality is attained if and only if there are no relations between the classes of (i) in $H^{\leq 2m}(M_{\beta,\chi})$. Thus the assumption (\ref{assumption_m}), which implies that (\ref{assumption_b}) is an equality, proves (i).

\smallskip

From (i) and a more refined monomial counting, we have 
\[
\dim\mathrm{Gr}^C_iH^{i+j}(M_{\beta,\chi}) = [H(q,t)]^{i,j}, \quad i+j \leq 2m.
\]
Therefore (\ref{assumption_m}) further yields
\[
\dim \mathrm{Gr}^P_{i}H^{i+j}(M_{\beta,\chi}) = \dim \mathrm{Gr}^C_i H^{i+j}(M_{\beta,\chi}).
\]
In particular, (ii) is a corollary of Theorem \ref{thm1.6}.
\end{proof}

In the next sections, we prove Theorem \ref{thm0.2} which fulfills \eqref{assumption_m}. This implies Theorem \ref{thm0.3} by Corollary \ref{cor1.7}. Theorems \ref{thm0.4} and \ref{mainthm} for the case $S=\mathbb{P}^2$ also follow immediately from the asymptotic bound we obtained; see Remark \ref{final1}.

\section{Relative Hilbert schemes}

We study the geometry of the relative Hilbert schemes associated with the linear system $|\beta|$ on the del Pezzo surface $S$. Proposition \ref{prop2.1} concerning their Betti numbers will play a crucial role in the proof of the main theorems.

\subsection{Betti numbers for relative Hilbert schemes}
Let $\CC \to |\beta|$ be the universal curve associated with the linear system.  Denote by $\CC_{|\beta|^\circ} \to |\beta|^\circ$ its restriction to the locus $|\beta|^\circ \subset |\beta|$ of integral curves. For an integer $k \geq 0$, we consider the relative Hilbert schemes of $k$ points:
\begin{equation}\label{Hilb}
\CC^{[k]} \to |\beta|, \quad \CC_{|\beta|^\circ}^{[k]} \to |\beta|^\circ.
\end{equation}
We define
\[
N_2(\beta):= \mathrm{min}\{N_1(\beta), -\beta \cdot K_S -1 \}.
\]
By Proposition \ref{prop1.4}, we see that $N_2(\beta)$ is also an asymptotic bound.

Our main result of this section is the following, which calculates low degree Betti numbers of the relative Hilbert scheme $\CC_{|\beta|^\circ}^{[k]}$ associated with integral curves in terms of the Betti numbers of Hilbert scheme $S^{[k]}$ of points on $S$.

\begin{prop}\label{prop2.1}
For $i\leq N_1(\beta)$ and $ k \leq N_2(\beta)$, we have
\begin{equation}\label{eq_2.1}
b_i\left(\CC_{|\beta|^\circ}^{[k]}\right) =   b_i\left( S^{[k]} \times \BP^{\:\!\dim |\beta|-k} \right).
\end{equation}
\end{prop}

\smallskip

When $S=\BP^2$, we see from Example \ref{ex} that
\begin{equation*}\label{bound1}
N_1(dH) = N_2(dH) = 2d-4.
\end{equation*}
We note that the identity (\ref{eq_2.1}) for the weaker\footnote{Since Conjecture \ref{conj1} is proved for $d\leq 5$ in \cite{KLMP},  we may assume $d\geq 6$ here and after.} bound $d+1$ is obvious. Indeed, since the line bundle $\CO_{\BP^2}(d)$ is $d$-very ample (see \textit{e.g.} \cite{BFS}), the evaluation map
\[
\mathrm{ev}: \CC^{[k]} \to S^{[k]}
\]
is a projective bundle of fiber dimension $\dim |\beta|-k$ when $k\leq d+1$; this implies that
\[
b_{i}\left(\CC^{[k]}_{|dH|^\circ}\right) = b_i(\CC^{[k]}) = b_i\left( S^{[k]} \times \BP^{\:\!\dim |\beta|-k} \right), \quad i,\; k\leq d+1,
\]
where the first identity is given by a codimension bound (see Lemma \ref{lem2.2} below). Using this bound obtained from $d$-very-ampleness, we deduce immediately from the argument of Section~\ref{sec3.3}:
\begin{equation*}
n^{i,j}_d = [H(q,t)]^{i,j}, \quad i+j \leq d+1,
\end{equation*}
which further implies the $P=C$ match:
\[
P_k H^{\leq d+1}(M_{d,\chi}) = C_kH^{\leq d+1}(M_{d,\chi})
\]
by Corollary \ref{cor1.7}. Ideas of using $k$-very-ampleness have been applied in enumerative geometry for surfaces; see for example the Kool--Shende--Thomas proof of the G\"ottsche conjecture \cite{KST}, and the work of \cite{CKK, Zhao} which carried out some explicit calculations of refined BPS invariants for the local surfaces $\BP^2$ and $\BP^1\times \BP^1$.

Improving the bound $d+1$ to the optimal bound $2d-4$ requires a careful study of the relative Hilbert schemes (\ref{Hilb}). In particular, we need to consider the case where $\CC^{[k]}$ is possibly singular, and the evaluation map $\mathrm{ev}: \CC^{[k]} \to S^{[k]}$ fails to be a projective bundle.

\subsection{Proof of Proposition \ref{prop2.1}}
\label{sec2.2}
We consider the locus $U \subset S^{[k]}$ which parameterizes 0-dimensional subschemes $Z\subset S$ of length $k$, satisfying that the natural morphism
\begin{equation}\label{restriction}
H^0(S, \CO_S(\beta)) \to H^0(S, \CO_S(\beta)|_Z)
\end{equation}
is surjective. The restriction of the evaluation map
\[
\mathrm{ev}: \CH_U := \mathrm{ev}^{-1}(U) \to U
\]
is a projective bundle over $U$ associated with the vector bundle
\[
\mathrm{ker}\left(H^0(S, \CO_S(\beta)) \twoheadrightarrow H^0(S, \CO_S(\beta)|_Z)\right) \mapsto [Z] \in U.
\]

\begin{prop}\label{prop2.2}
For $k\leq N_2(\beta)$, we have
\[
\CC_{|\beta|^\circ}^{[k]} \subset \CH_U \subset \CC^{[k]}.
\]
\end{prop}

\begin{proof}
It suffices to show that the restriction map (\ref{restriction}) is surjective if $Z\subset S$ is a length $k$ subscheme lying on an integral curve $C\subset S$ in the linear system $|\beta|$.

Note that the morphism (\ref{restriction}) is the composition of the two morphisms:
\[
H^0(S, \CO_S(\beta)) \xrightarrow{~\mathrm{(A)}~} H^0(S, \CO_S(\beta)|_C) \xrightarrow{~\mathrm{(B)}~} H^0(S, \CO_S(\beta)|_Z).
\]
We show that both (A) and (B) are in fact surjective.

To treat (A), we consider the short exact sequence
\[
0 \to \CO_S \to \CO_S(\beta) \to \CO_S(\beta)|_C \to 0
\]
whose associated long exact sequence reads
\[
\cdots \to H^0(S, \CO_S(\beta)) \xrightarrow{~\mathrm{(A)}~} H^0(S, \CO_S(\beta)|_C) \to H^1(S, \CO_S) \to \cdots.
\]
The surjectivity of (A) follows from the vanishing $H^1(S, \CO_S) = 0$. 

Now we consider (B). The relevant short exact sequence for the integral curve $C$ is 
\begin{equation}\label{SES2}
0 \to \CI_Z \otimes \CO_C(\beta) \to \CO_C(\beta) \to \CO_C(\beta)|_Z \to 0.
\end{equation}
By the assumption on $k$, the degree of the torsion-free sheaf $\CI_Z\otimes \CO_C(\beta)$ satisfies
\[
\mathrm{deg}(\CI_Z\otimes \CO_C(\beta)) = \beta^2-k\geq \beta^2- N_2(\beta) > \beta( \beta+ K_S) =\mathrm{deg}(\omega_{C}).
\]
By Serre duality and the stability of $\CI_Z\otimes \CO_C(\beta)$, this yields
\[
H^1(C, \CI_Z \otimes \CO_C(\beta)) \simeq \mathrm{Hom}_C\left(\CI_Z\otimes \CO_C(\beta), \omega_C\right)^\vee =0.
\]
The surjectivity of (B) then follows from the long exact sequence associated with (\ref{SES2}). This completes the proof.
\end{proof}

The following proposition concerns the geometry of the relative Hilbert schemes (\ref{Hilb}). 

\begin{prop}\label{prop2.3}
    For any $k \in \BN$, we have the following.
    \begin{enumerate}
        \item[(i)] The relative Hilbert scheme $\CC^{[k]}$ is irreducible.
        \item[(ii)] The open subset $\CC^{[k]}_{|\beta|^\circ} \subset \CC^{[k]}$ is nonsingular and Zariski dense.
        \item[(iii)] We have
        \[
        \dim \CC^{[k]} = \dim \CC_{|\beta|^\circ}^{[k]} = \frac{1}{2}\beta(\beta-K_S) +k.
        \]
        \item[(iv)] We have
        \[
        \mathrm{codim}_{\CC^{[k]}}\left(\CC^{[k]} \setminus \CC^{[k]}_{|\beta|^\circ}\right) \geq \frac{1}{2}N_1(\beta) +1. 
        \]
    \end{enumerate}
\end{prop}

\begin{proof}
By \cite[Proposition 14]{Shende2}, the relative Hilbert scheme $\CC^{[k]}_{|\beta|^\circ}$ is nonsingular of dimension
\begin{equation}\label{dim0}
\dim \CC^{[k]}_{|\beta|^\circ} = \frac{1}{2}\beta(\beta-K_S) +k = \dim |\beta| +k.
\end{equation}
Now we consider the natural morphism
\[
\CC^{[k]} \to |\beta|,
\]
whose every closed fiber has dimenison $k$ by \cite[Theorem 1.1]{Luan}. Combined with (\ref{dim0}), we know that if $\CC^{[k]}$ is not irreducible, then it must have a component of dimension $< \dim |\beta| +k$. 

On the other hand, if we view $\CC^{[k]}$ as the moduli space of Pandharipande--Thomas stable pairs \cite[Proposition B.8]{PT2}, it admits a two-term perfect obstruction theory of virtual dimension
\begin{equation}\label{vd}
\mathrm{vdim} = \dim \mathsf{Tan} - \dim \mathsf{Obs} =  \frac{1}{2}\beta(\beta-K_S) +k;
\end{equation}
see \cite[Theorem A.7]{KT}. Hence every irreducible component of $\CC^{[k]}$ has dimension at least (\ref{vd}), which is a contradiction. Therefore we obtain the irreducibility of $\CC^{[k]}$. All the three statements (i, ii, iii) then follow. (iv) is a direct consequence of \cite[Theorem 1.1]{Luan}.
\end{proof}

Now we prove Proposition \ref{prop2.1}. We note the following lemma.

\begin{lem}\label{lem2.2}
Let $X$ be a complete and irreducible variety of dimension $a$. Let $W\subset X$ be a nonsingular open subset. Assume that
\[
2\cdot\mathrm{codim}_X(X\setminus W) > i+1.
\]
Then we have
\[
b_i(W) = b_{2a-i}(X).
\]
If we further assume that $X$ is nonsingular, then
\[
b_i(W) = b_{i}(X).
\]
\end{lem}

\begin{proof}
We write $Z:= X\setminus W$. The excision sequence reads
\[
\cdots \to H^{2a-i-1}(Z) \to H_c^{2a-i}(W) \to H^{2a-i}(X) \to H^{2a-i}(Z) \to \cdots. 
\]
By the codimension assumption, we have
\[
2\cdot \dim Z < 2a -i-1
\]
which implies the vanishing
\[
 H^{2a-i-1}(Z) = H^{2a-i}(Z) =0.
 \]
Consequently, combined with Poincar\'e duality we obtain
\[
b_i(W)= \dim H_c^{2a-i}(W) = \dim H^{2a-i}(X) = b_{2a-i}(X). 
\]
This proves the first part. The second part follows directly from Poincar\'e duality.
\end{proof}

\begin{proof}[Proof of Proposition \ref{prop2.1}]
We fix an integer $k \leq N_2(\beta)$. By Proposition \ref{prop2.2}, we have
\[
\mathrm{codim}_{\CC^{[k]}}(\CC^{[k]} \setminus \CH_U) \geq \mathrm{codim}_{\CC^{[k]}}\left(\CC^{[k]}\setminus \CC^{[k]}_{|\beta|^\circ}\right) \geq \frac{1}{2}N_1(\beta)+1 
\]
where the last inequality is given by Proposition \ref{prop2.3} (iv). This also implies that 
\[
\mathrm{codim}_{S^{[k]}} (S^{[k]} \setminus U) \geq \frac{1}{2}N_1(\beta)+1
\]
due to the semi-continuity of the dimensions of the fibers of $\mathrm{ev}: \CC^{[k]} \to S^{[k]}$. By Lemma \ref{lem2.2}, we have
\begin{equation}\label{4_id}
b_i(U) =  b_i(S^{[k]}), \quad i \leq N_1(\beta).
\end{equation}
Moreover, the fiber dimension of the projective bundle $\mathcal{H}_U \to U$ equals
\[
\dim \CC^{[k]} - \dim S^{[k]} = \dim |\beta| - k
\]
by Proposition \ref{prop2.3} (iii). Therefore, we conclude for any $i\leq N_1(\beta)$ that
\begin{align*}
    b_i\left(\CC_{|\beta|^\circ}^{[k]}\right) & = b_{2\cdot\dim \CC^{[k]}-i}(\CC^{[k]}) 
     = b_{i}(\CH_U) \\
    & =b_i\left( U \times \BP^{\:\!\dim |\beta| - k} \right)\\
    & = b_i\left( S^{[k]} \times \BP^{\:\!\dim |\beta| - k} \right).
\end{align*}
Here the first and the second equalities follow from Lemma \ref{lem2.2}, the third follows from the definition of $\CH_U$, and the fourth is given by (\ref{4_id}).
\end{proof}

\section{Proof of the main theorems}\label{sec3}

The purpose of this section is to introduce the asymptotic bound $N(\beta)$, to calculate $n^{i,j}_d$ when $i+j \leq N(\beta)$, and to match the result with the closed formula given by $H(q,t)$. This proves Theorem \ref{thm0.2}. In view of Remark \ref{final1}, this also proves Theorem \ref{thm0.4} concerning the optimal bound for $\BP^2$. Theorems~\ref{thm0.3} and \ref{mainthm} then follow from Corollary \ref{cor1.7}.

\subsection{Compactified Jacobians}
Recall the (twisted) compactified Jacobian fibration 
\[
h^\circ: \overline{J}^e_{\CC_{|\beta|^\circ}} \to |\beta|^\circ
\]
obtained from the base change of $h: M_{\beta,\chi} \to |\beta|$ to the open subset $|\beta|^\circ \subset |\beta|$. Since the restriction map 
\[
H^{\leq {N_1(\beta)}}(M_{\beta,\chi}) \to H^{\leq N_1(\beta)}(\overline{J}^e_{\CC_{|\beta|^\circ}})
\]
is an isomorphism preserving the perverse filtrations \cite[Corollary 5.3]{MSY}, we have
\begin{equation}\label{n=h}
n_\beta^{i,j} =\dim \mathrm{Gr}^P_iH^{i+j}\left(\overline{J}^e_{\CC_{|\beta|^\circ}}\right), \quad i+j\leq N_1(\beta).
\end{equation}

Our final key ingredient is the following result of Maulik--Yun \cite{MY} and Migliorini--Shende \cite{MiSh}, which relates the perverse filtration associated with $h^\circ$ to the cohomology of the relative Hilbert scheme of points $\CC_{|\beta|^\circ}^{[k]}$ by a support theorem.

\begin{thm}[\cite{MY, MiSh}]\label{thm2.3}
For any $m\geq 0$, we have an isomorphism of vector spaces
\[
H^m\left(\CC_{|\beta|^\circ}^{[k]}\right) \simeq \bigoplus_{i+j \leq k,\;j\geq0} \mathrm{Gr}^P_{i}H^{m-2j}\left(\overline{J}^e_{\CC_{|\beta|^\circ}}\right).
\]
\end{thm}

In particular, we obtain from (\ref{n=h}) that
\begin{equation}\label{betti}
b_m\left(\CC^{[k]}_{|\beta|^\circ}\right) = \sum_{i+j\leq k,\;j\geq0} n_{\beta}^{i,m-i-2j}, \quad m \leq N_1(\beta).
\end{equation}

\subsection{Proof of Theorem \ref{thm0.2}}
\label{sec3.3}

We first introduce the bound for Theorem \ref{thm0.2}:
\begin{equation}\label{final_bd}
N(\beta):= \min\left\{N_2(\beta),~\beta(\beta+K_S)+2 \right\} = \min\left\{N_1(\beta),~ -\beta\cdot K_S-1, ~\beta(\beta+K_S)+2 \right\}.
\end{equation}
This is clearly an asymptotic bound by Proposition \ref{prop1.4}; see also Remark \ref{final1}.


We proceed by induction on $i$ to prove that
\[
n_\beta^{i,j} = [H(q,t)]^{i,j}, \quad i+j \leq N(\beta).
\]
We only consider $n_\beta^{i,j}$ in this range. For $i=0$, a direct calculation or using \cite[Proposition 5.2.4]{dCM0} gives 
\[
n_\beta^{0,j} = \begin{cases}
    1 \quad j \textrm{ is even};\\
    0 \quad j \textrm{ is odd}.
\end{cases}
\]
On the other hand, we have $[H(q,t)]^{0,j}=1$ for  even $j$ and $0$ otherwise. Thus the case $i=0$ holds, which gives the induction base.

Assume now that Theorem \ref{thm0.2} holds for $i\leq \ell-1$ where $\ell> 0$ is an integer satisfying $\ell-1 < N(\beta)$. This means that we have the numerical match 
\[
n_\beta^{i,j} = [H(q,t)]^{i,j}, \quad i\leq \ell-1,\quad i+j\leq N(\beta).
\]
We first show for $k\leq N_1(\beta)$ that
\begin{equation}\label{bettidiff}
    b_k\left({\CC}_{|\beta|^\circ}^{[\ell]}\right)-b_k\left({\CC}_{|\beta|^\circ}^{[\ell-1]}\right)= \sum_{i+j=\ell} [H(q,t)]^{i, k-i-2j}.
\end{equation}

The left-hand side is governed by Proposition \ref{prop2.1}. If $k$ is odd, then both sides are zero and clearly (\ref{bettidiff}) holds. So we only focus on the case where $k$ is even. In view of the formula of Proposition \ref{prop2.3} (iii) and the asymptotic bound (\ref{final_bd}), we have 
\[
\dim \CC^{[\ell]} - \dim S^{[\ell]} = \dim |\beta|-\ell \geq  \frac{k}{2}. 
\]
Therefore Proposition~\ref{prop2.1} implies that
\[
b_k(\CC_{|\beta|^\circ}^{[\ell]}) = \sum_{s\leq k} b_s(S^{[\ell]}), \quad k\leq N_1(\beta).
\]
This also holds if we replace $\ell$ with $\ell-1$. Consequently, we have
\begin{equation}\label{b=b}
b_k\left({\CC}_{|\beta|^\circ}^{[\ell]}\right)-b_k\left({\CC}_{|\beta|^\circ}^{[\ell-1]}\right) = \sum_{s=0}^{k} b_{s}({S}^{[\ell]})-b_{s}({S}^{[\ell-1]}),\quad k \leq N_1(\beta).
\end{equation}
We will need the Göttsche formula \cite{Go1} (in two variables)
\[
G(z,w):=\prod_{i\geq 1} \frac{1}{(1-z^{2i-2}w^i)(1-z^{2i} w^i)^\rho(1-z^{2i+2} w^i)},
\]
which calculates Betti numbers of ${S}^{[m]}$: 
\[
G(z,w) = \sum_{m\geq 0}\sum_{s\geq 0} b_{s}({S}^{[m]})z^s w^m.
\]
Note that $G(z,w)$ specializes to the asymptotic version \eqref{Gott} via $z=q$ and $w=1$. Combining this and (\ref{b=b}), we see that \eqref{bettidiff} is equivalent to
\begin{equation}
\label{coeffidentity}
    \sum_{s=0}^{k}([G(z,w)]^{s,\ell} - [G(z,w)]^{s,\ell-1}) = \sum_{i+j=\ell} [H(q,t)]^{i,k-i-2j}.
\end{equation}
Now we compare the two sides:
\begin{align*}
    \textrm{LHS of } \eqref{coeffidentity} & = \sum_{s=0}^{k}[G(z,w)\cdot (1-w)]^{s,\ell} = \left[G(z,w)\cdot\frac{1-w}{1-z^2}\right]^{k,\ell},\\
    \textrm{RHS of }\eqref{coeffidentity} & = \sum_{i=0}^\ell [H(q,t)]^{i,k-2\ell+i}=\left[H(q,t)\cdot \frac{1}{1-qt}\right]^{\ell, k-\ell}.
\end{align*}
The identity \eqref{coeffidentity} then follows from the observation that the two series
\[
G(z,w)\cdot\frac{1-w}{1-z^2}, \quad H(q,t)\cdot\frac{1}{1-qt}
\]
are identified by the change of variables 
\[
z=t,\quad  w =q/t.
\]
This completes the proof of (\ref{bettidiff}).

Now we finish the induction. Note that the left-hand side of (\ref{bettidiff}) is equal to 
\[
\sum_{i+j=\ell} n_\beta^{i,k-i-2j} = n_\beta^{\ell, k-\ell} + \sum_{i<\ell,\, i+j=\ell}   n_\beta^{i,k-i-2j}
\]
by combining \eqref{betti} and \eqref{n=h}. Since \eqref{bettidiff} holds, using the induction hypothesis to remove terms with $i\leq \ell-1$ on both sides yields \[
n_\beta^{\ell, k-\ell} = [H(q,t)]^{\ell, k-\ell}, \quad k \leq N(\beta).
\]
This completes the induction step, and hence the proof of Theorem \ref{thm0.2}. 
\qed

\begin{rmk}\label{final1}
    One can calculate using results in \cite{Rocco} that
    \[
    N(\beta) = N_1(\beta)= 2\cdot \mathrm{codim}_{|\beta|}(|\beta| \setminus |\beta|^\circ)-2 
    \]
    for del Pezzo surfaces of degrees $\geq 3$, namely $\mathbb{P}^2,\: \mathbb{P}^1\times \mathbb{P}^1$, and the blow-up of $\mathbb{P}^2$ in $n$ very general points where $1\leq n \leq 6$.\footnote{For $n=7,8$, the numerical analysis is more involved and the statement fails for certain curve classes $\beta$.}
    In particular, for $S=\BP^2$ we have
    \[
    N(dH) = \min\left\{2d-4,\; d(d-3)+2\right\} = 2d-4, \quad d\geq 1.
    \]
Therefore our asymptotic bound recovers the optimal bound for $\BP^2$, and consequently Theorem~\ref{thm0.2} implies Theorems \ref{thm0.4} and \ref{mainthm}.

\smallskip

Since $\beta(\beta+K_S)+2$ is quadratically dependent on the curve class $\beta$, it is clear that 
\[
N(m\beta) = N_2(m\beta)
\]
when $m$ is sufficiently large.
\end{rmk}

\subsection{Final remarks} We conclude the paper with a few further remarks.

\subsubsection{Cohomology of perverse sheaves and stabilization}\label{3.4.1}

The invariants $n_\beta^{i,j}$ only depend on nonsingular curves in the linear system $|\beta|$. More precisely, let $V\subset |\beta|$ be the open subset of nonsingular curves in $S$, with $f_V: \CC_V \to V$ the universal curve. Then we can obtain a natural local system
\[
\CL_\beta:= R^1f_{V*} \BQ_{\CC_V} \in \mathrm{LocSys}(V)
\]
which calculates the cohomology $H^1$ of the fibers. By the full support theorem \cite{MS_GT}, the refined BPS invariants can be computed as the cohomology
\begin{equation}\label{LS}
n_\beta^{i,j}:= \dim H^{j-\dim {|\beta|}}\left( |\beta|, \mathrm{IC}(\wedge^i \CL_\beta)\right).
\end{equation}
It would be interesting to find an effective algorithm to calculate the cohomology groups on the right-hand side directly. Our main result implies the stabilization of $n^{i,j}_\beta$ but our proof is \emph{indirect}. A natural question is to find a direct explanation of this stabilization phenomenon for the cohomology of perverse sheaves. 

\subsubsection{Formulas for refined BPS invariants associated with surfaces}

Although the (unrefined) BPS invariants have been calculated in many cases via Gromov--Witten/Donaldson--Thomas theory over the last few decades, very few closed formulas are known for the refined BPS invariants.

In a few special cases, there are closed formulas for all the refined BPS invariants. For the case of a local K3 surface, it was proven in \cite{SY} that the refined BPS invariants match perfectly with the Hodge numbers of the Hilbert scheme of points on a K3 surface, as predicted by \cite{KKV, KKP}. We refer to \cite{T} for other refined invariants associated with K3 surfaces.

For the local curve $X= T^*C \times \BC$ with $C$ a curve of genus $g\geq 2$, the geometry of BPS invariants are closely related to the topology of the Hitchin system. As a consequence of the $P=W$ conjecture (now a theorem), the refined BPS invariants of $X$ are given by the weight polynomials of the character variety associated with the curve $C$, whose closed formulas were conjectured by \cite{HRV} in terms of Macdonald polynomials. We refer to \cite{CDP} for more details on this connection.

Recently, Oberdieck conjectured in certain curve classes closed formulas for the local Enriques surface \cite{Ober2}, refining the Gromov--Witten calculation \cite{Ober}.

All these formulas have asymptotic product expressions whose shapes are similar to $H(q,t)$.

\end{document}